\documentclass[10pt,a4paper]{amsart}
\usepackage{amsmath, amssymb, amsfonts, amsthm, float, stmaryrd, epsfig}
\usepackage[pdfborder={0 0 0 [0 0 ]},bookmarksdepth=3]{hyperref}
\usepackage[latin1]{inputenc}
\usepackage[capitalize]{cleveref}
\usepackage{color}
\usepackage[ps,all,arc,rotate]{xy}
\usepackage{fancyhdr}

\hypersetup{
    colorlinks=true,
    linkcolor=black,
    citecolor=black,
    filecolor=black,
    urlcolor=black,
}

\numberwithin{figure}{section}

\theoremstyle{plain}
\newtheorem{thm}{Theorem}[section]
\newtheorem*{prop*}{Proposition}
\newtheorem*{thm*}{Theorem}
\newtheorem{prop}[thm]{Proposition}
\newtheorem{lem}[thm]{Lemma}
\newtheorem{cor}[thm]{Corollary}

\theoremstyle{definition}

\newtheorem{dfn}[thm]{Definition}
\newtheorem*{dfn*}{Definition}

\theoremstyle{remark}
\newtheorem{rmk}[thm]{Remark}

\newtheoremstyle{maintheorem}{}{}{\itshape}{}{\bfseries}{}{.5em}{#1 \!\thmnote{#3}.}
\theoremstyle{maintheorem}
\newtheorem*{mainthm}{Theorem}

\newcommand{\Aut}{\operatorname{Aut}}
\newcommand{\SAut}{\operatorname{SAut}}
\newcommand{\Out}{\operatorname{Out}}
\newcommand{\SOut}{\operatorname{SOut}}
\newcommand{\GL}{\operatorname{GL}}
\newcommand{\SL}{\operatorname{SL}}

\newcommand{\PGL}{\operatorname{PGL}}
\newcommand{\PSL}{\operatorname{PSL}}

\newcommand{\st}{\operatorname{st}}
\newcommand{\lk}{\operatorname{lk}}

\newcommand{\Alt}{\operatorname{Alt}\nolimits}
\newcommand{\Sym}{\operatorname{Sym}\nolimits}

\def\C{\mathbb{C}}
\def\R{\mathbb{R}}

\def\Z{\mathbb{Z}}

\def\K{\mathbb{K}}

\def\s-{\smallsetminus}
\def\into{\hookrightarrow}

\def\iff{if and only if }

\newcounter{dawidcomments}

\author{Dawid Kielak}
\title[Outer actions of $\Out(F_n)$ on small RAAGs]{Outer actions of $\Out(F_n)$ on small right-angled Artin groups}
\date{\today}

\begin{document}

\begin{abstract}
 We determine the precise conditions under which $\SOut(F_n)$, the unique index two subgroup of $\Out(F_n)$, can act non-trivially via outer automorphisms on a RAAG whose defining graph has fewer than $\frac 1 2 \binom n 2 $ vertices.

We also show that the outer automorphism group of a RAAG cannot act faithfully via outer automorphisms on a RAAG with a strictly smaller (in number of vertices) defining graph.

Along the way we determine the minimal dimensions of non-trivial linear representations of congruence quotients of the integral special linear groups over algebraically closed fields of characteristic zero, and provide a new lower bound on the cardinality of a set on which $\SOut(F_n)$ can act non-trivially.
\end{abstract}

\maketitle

\section{Introduction}

The main purpose of this article is to study the ways in which $\Out(F_n)$ can act via outer automorphisms on a right-angled Artin group $A_\Gamma$ with defining graph $\Gamma$. (Recall that $A_\Gamma$ is given by a presentation with generators being the vertices of $\Gamma$, and relators being commutators of vertices which span an edge in $\Gamma$.)  
Such actions have previously been studied for the extremal cases: when the graph $\Gamma$ is discrete, we have $\Out(A_\Gamma) = \Out(F_m)$ for some $m$, and homomorphisms
\[ \Out(F_n) \to \Out(F_m) \]
have been investigated by Bogopolski--Puga~\cite{BogopolskiPuga2002}, Khramtsov~\cite{Khramtsov1990}, Bridson--Vogtmann~\cite{BridsonVogtmann2012}, and the author~\cite{Kielak2013, Kielak2015a}. When the graph $\Gamma$ is complete, we have $\Out(A_\Gamma) = \GL_m(\Z)$, and homomorphisms
\[\Out(F_n) \to \GL_m(\Z) \]
or more general representation theory of $\Out(F_n)$
have been studied by Grunewald--Lubotzky~\cite{GrunewaldLubotzky2006}, Potapchik--Rapinchuk~\cite{PotapchikRapinchuk2000}, Turchin--Wilwacher~\cite{TurchinWillwacher2015}, and the author~\cite{Kielak2013, Kielak2015a}.


\smallskip
There are two natural ways of constructing non-trivial homomorphisms
\[\phi \colon \Out(F_n) \to \Out(A_\Gamma)\]
When $\Gamma$ is a join of two graphs, $\Delta$ and $\Sigma$ say, then $\Out(A_\Gamma)$ contains 
\[\Out(A_\Delta) \times \Out(A_\Sigma)\]
 as a finite index subgroup. When additionally $\Delta$ is isomorphic to the discrete graph with $n$ vertices, then $\Out(A_\Delta) = \Out(F_n)$, and so we have an  obvious embedding $\phi$.

In fact this method works also for a discrete $\Delta$ with a very large number of vertices, since there are injective maps $\Out(F_n) \to \Out(F_m)$ constructed by Bridson--Vogtmann~\cite{BridsonVogtmann2012} for specific values of $m$ growing exponentially with $n$.

\smallskip
The other way of constructing non-trivial homomorphisms $\phi$ becomes possible when $\Gamma$ contains $n$ vertices with identical stars. In this case it is immediate that these vertices form a clique $\Theta$, and we have a map
\[ \GL_n(\Z) = \Aut(A_\Theta) \to \Aut(A_\Gamma) \to \Out(A_\Gamma) \]
We also have the projection
\[ \Out(F_n) \to \Out(H_1(F_n)) = \GL_n(\Z) \]
and combining these two maps gives us a non-trivial (though also non-injective) $\phi$.

This second method does not work in other situations, due to the following result of Wade.
\begin{thm}[{\cite{Wade2013}}]
Let $n \geqslant 3$.
Every homomorphism
\[\SL_n(\Z) \to \Out(A_\Gamma)\]
 has finite image \iff $\Gamma$ does not contain $n$ distinct vertices with equal stars.
\end{thm}
In fact Wade proved a much more general result, in which the domain of the homomorphism is allowed to be any irreducible lattice in a real semisimple Lie group with finite centre and without compact factors, and with real rank $n-1$.

The aim of this paper is to prove
\begin{mainthm}[\ref{main thm}]
Let $n \geqslant 6$. Suppose that $\Gamma$ is a simplicial graph with fewer than $\frac 1 2 \binom n 2$ vertices, which does not
contain $n$ distinct vertices with equal stars, and is not a join of the discrete graph with $n$ vertices and another (possibly empty) graph. Then
every homomorphism $\SOut(F_n) \to \Out(A_\Gamma)$ is trivial.
\end{mainthm}
Here $\SOut(F_n)$ denotes the unique index two subgroup of $\Out(F_n)$.

The proof is an induction, based on an observation present in a paper of Charney--Crisp--Vogtmann~\cite{Charneyetal2007}, elaborated further in a paper of Hensel and the author~\cite{HenselKielak2016}, which states that, typically, the graph $\Gamma$ contains many induced subgraphs $\Sigma$ which are \emph{invariant up to symmetry}, in the sense that the subgroup of $A_\Gamma$ the vertices of $\Sigma$ generate is invariant under any outer action up to an automorphism induced by a symmetry of $\Gamma$ (and up to conjugacy).

To use the induction we need to show that such subgraphs are really invariant, that is that we do not need to worry about the symmetries of $\Gamma$. To achieve this we prove
\begin{mainthm}[\ref{thm: action on finite set}]
 Every action of $\Out(F_n)$ (with $n \geqslant 6$) on a set of cardinality $m \leqslant \binom {n+1} 2$ factors through $\Z/ 2 \Z$.
\end{mainthm}
Since $\SOut(F_n)$ is the unique index two subgroup of $\Out(F_n)$, the conclusion of this theorem is equivalent to saying that $\SOut(F_n)$ lies in the kernel of the action.

A crucial ingredient in the proof of this theorem is the following.
\begin{mainthm}[\ref{thm: reps of GLn(Zq)}]
Let $V$ be a non-trivial, irreducible $\K$-linear representation of \[\mathrm{SL}_n(\Z/q\Z)\] where $n \geqslant 3$, $q$ is a power of a prime $p$, and where $\K$ is an algebraically closed field of characteristic 0. Then
\[\dim V \geqslant \left\{ \begin{array}{ccc} 2 & \textrm{ if } & (n,p) = (3,2) \\ p^{n-1} -1 & \textrm{ otherwise } \end{array} \right.\]
\end{mainthm}
This result seems not to be present in the literature; it extends a theorem of Landazuri--Seitz~\cite{LandazuriSeitz1974} yielding a very similar statement for $q=p$ (see \cref{thm: landazuri and seitz}).

\medskip

At the end of the paper we also offer
\begin{mainthm}[\ref{main fewer vertices}]
 There are no injective homomorphisms $\Out(A_\Gamma) \to \Out(A_{\Gamma'})$ when $\Gamma'$ has fewer vertices than $\Gamma$.
\end{mainthm}

This theorem follows from looking at the $\Z/2\Z$-rank, i.e. the largest subgroup isomorphic to $(\Z/2\Z)^k$.

\section{The tools}

\subsection{Automorphisms of free groups}

\begin{dfn}[$\SOut(F_n)$]
Consider the composition
\[ \Aut(F_n) \to \GL_n(\Z) \to \Z/2\Z\]
where the first map is obtained by abelianising $F_n$, and the second map is the determinant. We define $\SAut(F_n)$ to be the kernel of this map; we define $\SOut(F_n)$ to be the image of $\SAut(F_n)$ in $\Out(F_n)$.
\end{dfn}

It is easy to see that both $\SAut(F_n)$ and $\SOut(F_n)$ are index two subgroups of, respectively, $\Aut(F_n)$ and $\Out(F_n)$.

The group $\SAut(F_n)$ has a finite presentation given by Gersten~\cite{Gersten1984}, and from this presentation one can immediately obtain the following result.

\begin{prop}[Gersten~{\cite{Gersten1984}}]
\label{prop: abelianise sout}
The abelianisation of $\SAut(F_n)$, and hence of $\SOut(F_n)$, is trivial for all $n \geqslant 3$.
\end{prop}
It follows that $\SOut(F_n)$ is the unique subgroup of $\Out(F_n)$ of index two. 

\smallskip
We will now look at symmetric and alternating subgroups of $\Out(F_n)$, and list some corollaries of their existence.

\begin{prop}[{\cite[Proposition 1]{bridsonvogtmann2003}}]
\label{prop: A_n in SOut original}
Let $n \geqslant 3$.
There exists a symmetric subgroup of rank $n$ \[\Sym_n < \Out(F_n)\] such that
any homomorphism $\phi \colon \Out(F_n) \to G$ that is not injective on $\Sym_n$ has image of cardinality at most 2.
\end{prop}

The symmetric group is precisely the symmetric group operating on some fixed basis of $F_n$. It is easy to see that it intersects $\SOut(F_n)$ in an alternating group $\Alt_n$. Whenever we talk about the alternating subgroup $\Alt_n$ of $\SOut(F_n)$, we mean this subgroup. Note that $\SOut(F_n)$ actually contains an alternating subgroup of rank $n+1$, which is a supergroup of our $\Alt_n$; we will denote it by $\Alt_{n+1}$. There is also a symmetric supergroup $\Sym_{n+1}$ of $\Alt_{n+1}$ contained in $\Out(F_n)$.

The proof of \cite[Proposition 1]{bridsonvogtmann2003} actually allows one to prove the following proposition.

\begin{prop}
\label{prop: A_n in SOut}
Let $n \geqslant 3$.
Then $\SOut(F_n)$ is the normal closure of any non-trivial element of $\Alt_n$.
\end{prop}

Following the proof of~\cite[Theorem A]{bridsonvogtmann2003}, we can now conclude

\begin{cor}
Let
\[\phi \colon \SOut(F_n) \to \GL_k(\Z)\]
 be a homomorphism, with $n \geqslant 6$ and $k< n $. Then $\phi$ is trivial.
 \label{thm: maps from out to small gl}
\end{cor}
\begin{proof}
For $n \geqslant 6$, the alternating group $\Alt_{n+1}$ does not have non-trivial complex representations below dimension $n$. 
Thus $\phi\vert_{\Alt_{n+1}}$ is not injective, and therefore trivial, as $\Alt_{n+1}$ is simple. Now we apply \cref{prop: A_n in SOut}.
\end{proof}

More can be said about linear representations of $\Out(F_n)$ in somewhat larger dimensions -- see~\cite{Kielak2013, Kielak2015a, TurchinWillwacher2015}.

\smallskip
Another related result that we will use is the following.
\begin{thm}[\cite{Kielak2013}]
\label{thm: main phd}
Let $n \geqslant 6$ and $m< \binom n 2$. Then every homomorphism $\Out(F_n) \to \Out(F_m)$ has image of cardinality at most 2, provided that $m \neq n$.
\end{thm}

In fact, we will need to go back to the proof of the above theorem and show:
\begin{thm}
\label{thm: main phd for sout}
Let $n \geqslant 6$ and $m< \frac 1 2 \binom n 2$. Then every homomorphism \[\SOut(F_n) \to \Out(F_m)\] is trivial, provided that $m \neq n$.
\end{thm}
The proof of this result forms the content of the next section.

\subsection{Homomorphisms $\SOut(F_n) \to \Out(F_m)$}

To study such homomorphisms we need to introduce finite subgroups $B_n$ and $B$ of $\SOut(F_n)$ that will be of particular use. Let $F_n$ be freely generated by $\{ a_1, \dots, a_n \}$.

\begin{dfn}
\label{defaxi}
Let us define $\delta \in \Out(F_n)$ by $\delta(a_i) = {a_i}^{-1}$ for each $i$. (Formally speaking, this defines an element in $\Aut(F_n)$; we take $\delta$ to be the image of this element in $\Out(F_n)$.)
Define $\sigma_{12} \in \Sym_{n} < \Out(F_n)$ to be the transposition swapping $a_1$ with $a_2$.
Define $\xi \in \SOut(F_n)$ by 
\begin{displaymath}
\xi = \left\{ \begin{array}{cl}  \delta  & \textrm{ if } n \textrm{ is even} \\
\delta \sigma_{12}  & \textrm{ if } n \textrm{ is odd}
\end{array} \right.
\end{displaymath}
and set $B_n = \langle \Alt_{n+1}, \xi \rangle \leqslant \SOut(F_n)$.

We also set $A$ to be either $\Alt_{n-1}$, the pointwise stabiliser of $\{1,2\}$ when $\Alt_{n+1}$ acts on $\{1,2,\ldots,n+1\}$ in the natural way (in the case of odd $n$), or $\Alt_{n+1}$ (in the case of even $n$). Furthermore, we set $B = \langle A ,\xi \rangle$.
\end{dfn}
It is easy to see that $B_n$ is a finite group -- it is a subgroup of the automorphism group of the (suitably marked) \emph{$(n+1)$-cage graph}, that is a graph with 2 vertices and $n+1$ edges connecting one to another. 

To prove \cref{thm: main phd for sout} we need to introduce some more notation from~\cite{Kielak2013}. Throughout, when we talk about modules or representations, we work over the complex numbers.

\begin{dfn}
A $B$-module $V$ admits a \emph{convenient split} \iff $V$ splits as a $B$-module into
\[ V = U \oplus U' \]
where $U$ is a sum of trivial $A$-modules and $\xi$ acts as minus the identity on $U'$.
\end{dfn}

\begin{dfn}
A graph $X$ with a $G$-action is called \emph{$G$-admissible} \iff it is connected, has no vertices of valence 2, and any $G$-invariant forest in $X$ contains no edges. Here by `invariant' we mean setwise invariant. 
\end{dfn}

\begin{prop}[\cite{Kielak2013}]
\label{technical prop}
Let $n \geqslant 6$.
 Suppose that $X$ is a $B_n$-admissible graph of rank smaller than $\binom{n+1} 2$ such that
 \begin{enumerate}
  \item the $B$-module $H_1(X;\C)$ admits a convenient split; and
  \item any vector in $H_1(X;\C)$ which is fixed by $\Alt_{n+1}$ is also fixed by $\xi$; and
  \item the action of $B_n$ on $X$ restricted to $A$ is non-trivial.
 \end{enumerate}
Then $X$ is the $(n+1)$-cage.
\end{prop}

The above proposition  does not (unfortunately) feature in this form in~\cite{Kielak2013} -- it does however follow from the proof of \cite[Proposition 6.7]{Kielak2013}.

\begin{proof}[Proof of \cref{thm: main phd for sout}]
Let $\phi \colon \SOut(F_n) \to \Out(F_m)$ be a homomorphism.
Using Nielsen realisation for free groups (due to, independently, Culler~\cite{culler1984}, Khramtsov~\cite{khramtsov1985} and Zimmermann~\cite{zimmermann1981})  we construct a finite connected graph $X$ with fundamental group $F_m$, on which $B_n$ acts in a way realising the outer action $\phi\vert_{B_n}$. We easily arrange for $X$ to be $B_n$-admissible by collapsing invariant forests.
Note that $V = H_1(F_m;\C)$ is naturally isomorphic to $H_1(X;\C)$ as a $B_n$-module. 

We have a linear representation
\[ \SOut(F_n) \to \Out(F_m) \to \GL(V) \]
where the first map is $\phi$. We can induce it to a linear representation \[\Out(F_n) \to \GL(W)\] of dimension $\dim W = 2 \dim V = 2m$. Since we are assuming that
\[ m < \frac 1 2 \binom n 2 \]
the combination of~\cite[Lemma 3.8 and Proposition 3.11]{Kielak2013} tells us that $W$ splits as an $\Out(F_n)$-module as 
\[W = W_0 \oplus W_1 \oplus W_{n-1} \oplus W_n\]
where the action of $\Out(F_n)$  is trivial on $W_0$ but not on $W_n$, and the action of the subgroup $\SOut(F_n)$ is trivial on both.
Moreover, as $\Sym_{n+1}$ modules, $W_1$ is the sum of standard and $W_{n-1}$ of signed standard representations. We also know that $\delta$ acts on $W_i$ as multiplication by $(-1)^i$.

When $n$ is even this immediately tells us that, as a $B=B_n$-module, we have
\[ W = U \oplus U' \]
where $U = W_0 \oplus W_n$ is sum of trivial $A = \Alt_{n+1}$-modules, and $\xi = \delta$ acts on \[U' = W_1 \oplus W_{n-1}\] as minus the identity.

When $n$ is odd we can still write
\[ W = U \oplus U' \]
as a $B$-module, with $A$ acting trivially on $U$ and $\xi$ acting as minus the identity on $U'$. Here we have $W_0 \oplus W_n < U$, but $U$ also contains the trivial $A$-modules contained in $W_1 \oplus W_{n-1}$. The module $U'$ is the sum of the standard $A$-modules.
Thus $W$ admits a convenient split.

Now we claim that $V$ also admits a convenient split as a $B$-module. To define the induced $\Out(F_n)$-module $W$ we need to pick en element $\Out(F_n) \s- \SOut(F_n)$; we have already defined such an element, namely $\sigma_{12}$. The involution $\sigma_{12}$ commutes with $\xi$ and conjugates $A$ to itself. Thus, as an $A$ module, $V$ could only consist of the trivial and standard representations, since these are the only $A$-modules present in $W$. Moreover, any trivial $A$-module in $V$ is still a trivial $A$-module in $W$, and so $\xi$ acts as minus the identity on it. Therefore $V$ also admits a convenient split as a $B$-module.
This way we have verified assumption (1) of \cref{technical prop}.

Observe that the $\SOut(F_n)$-module $V$ embeds into $W$. In $W$  every $\Alt_{n+1}$-fixed vector lies in $W_0 \oplus W_n$, and here $\xi$ acts as the identity. Thus assumption (2) of \cref{technical prop} is satisfied in $W$, and therefore also in $V$.

We have verified the assumptions (1) and (2) of \cref{technical prop}; we also know that the conclusion of \cref{technical prop} fails, since the $n+1$-cage has rank $n$, which would force $m=n$, contradicting the hypothesis of the theorem. Hence we know that assumption (3) of \cref{technical prop} fails, and so $A$ acts trivially on $X$. But this implies that $A \leqslant \ker \phi$.

Note that $A$ is a subgroup of the simple group $\Alt_{n+1}$, and so we have \[\Alt_{n+1} \leqslant \ker \phi\] But then \cref{prop: A_n in SOut} tells us that $\phi$ is trivial.
\end{proof}

\subsection{Automorphisms of RAAGs}

Throughout the paper, $\Gamma$ will be a simplicial graph, and $A_\Gamma$ will be the associated RAAG, that is the group generated by the vertices of $\Gamma$, with a relation of two vertices commuting \iff they are joined by an edge in $\Gamma$.

We will often look at subgraphs of $\Gamma$, and we always take them to be induced subgraphs. Thus we will make no distinction between a subgraph of $\Gamma$ and a subset of the vertex set of $\Gamma$.

Given an induced subgraph $\Sigma \subseteq \Gamma$ we define $A_\Sigma$ to be the subgroup of $A_\Gamma$ generated by (the vertices of) $\Sigma$. Abstractly, $A_\Sigma$ is isomorphic to the RAAG associated to $\Sigma$ (since $\Sigma$ is an induced subgraph).

\begin{dfn}[Links, stars, and extended stars]
Given a subgraph $\Sigma \subseteq \Gamma$ we define
\begin{itemize}
\item $\lk(\Sigma) = \{ w \in \Gamma \mid w \textrm{ is adjacent to } v \textrm{ for all } v \in \Sigma \}$;
\item $\st(\Sigma) = \Sigma \cup \lk(\Sigma)$;
\item $\widehat \st(\Sigma) = \lk(\Sigma) \cup \lk(\lk(\Sigma))$.
\end{itemize}
\end{dfn}

\begin{dfn}[Joins and cones]
We say that two subgraphs $\Sigma, \Delta \subseteq \Gamma$ form a \emph{join} $\Sigma \ast \Delta \subseteq \Gamma$ \iff $\Sigma \subseteq \lk(\Delta)$ and $\Delta \subseteq \lk(\Sigma)$.

A subgraph $\Sigma \subseteq \Gamma$ is a \emph{cone} \iff there exists a vertex $v \in \Sigma$ such that $\Sigma = v \ast (\Sigma \s- \{v\})$. In particular, a singleton is a cone.
\end{dfn}

\begin{dfn}[Join decomposition]
Given a graph $\Sigma$ we say that
\[ \Sigma  = \Sigma_1 \ast \dots \ast \Sigma_k \]
is the \emph{join decomposition} of $\Sigma$ when each $\Sigma_i$ is non-empty, and is not a join of two non-empty subgraphs.

Each of the graphs $\Sigma_i$ is called a \emph{factor}, and the join of all the factors which are singletons is called the \emph{clique factor}.
\end{dfn}

We will often focus on a specific finite index subgroup $\Out^0(A_\Gamma)$ of $\Out(A_\Gamma)$, called the \emph{group of pure outer automorphisms of $A_\Gamma$}.
To define it we need to discuss a generating set of $\Out(A_\Gamma)$ due to Laurence~\cite{Laurence1995} (it was earlier conjectured to be a generating set by Servatius~\cite{Servatius1989}).

$\Aut(A_\Gamma)$ is generated by the following
classes of automorphisms:
\begin{enumerate}
\item Inversions
\item Partial conjugations
\item Transvections
\item Graph symmetries
\end{enumerate}
Here, an \emph{inversion} maps one generator of $A_\Gamma$ to its
inverse, fixing all other generators.

A \emph{partial conjugation} needs a vertex $v$; it conjugates all generators in one
connected component of $\Gamma \s- \st(v)$ by $v$, and fixes all other generators.

A \emph{transvection} requires vertices $v, w$ with $\st(v)\supseteq
\lk(w)$. For such $v$ and $w$, a transvection is the automorphism which maps
$w$ to $w v$, and fixes all other generators.


A \emph{graph symmetry} is an automorphism of $A_\Gamma$ which permutes
the generators according to a combinatorial automorphism of $\Gamma$.

\smallskip
The group $\Aut^0(A_\Gamma)$ of pure automorphisms is defined to be the subgroup generated by generators of the first three types, i.e. without graph symmetries. The group $\Out^0(A_\Gamma)$ of pure outer automorphisms is the quotient of $\Aut^0(A_\Gamma)$ by the inner automorphisms.

Let us quote the following result of Charney--Crisp--Vogtmann:

\begin{prop}[{\cite[Corollary 3.3]{Charneyetal2007}}]
\label{out0}
There exists a finite subgroup \[Q < \Out(A_\Gamma)\]
consisting solely of graph symmetries, such that
\[ \Out(A_\Gamma) = \Out^0(A_\Gamma) \rtimes Q \]
\end{prop}

\begin{cor}
\label{out0 for homs}
Suppose that any action of $G$ on a set of cardinality at most $k$ is trivial, and assume that $\Gamma$ has $k$ vertices. Then
any homomorphism
 \[ \phi \colon G \to \Out(A_\Gamma) \]
 has image contained in $\Out^0(A_\Gamma)$.
\end{cor}
\begin{proof}
\cref{out0} tells us that
\[ \Out(A_{\Gamma}) = \Out^0(A_{\Gamma}) \rtimes Q\]
for some group $Q$ acting faithfully on $\Gamma$. Hence we can postcompose $\phi$ with the quotient map
\[ \Out^0(A_{\Gamma}) \rtimes Q \to Q \]
and obtain an action of $G$ on the set of vertices of $\Gamma$. By assumption this action has to be trivial, and thus $\phi(G)$ lies in the kernel of this quotient map, which is $\Out^0(A_{\Gamma})$.
\end{proof}

\begin{dfn}[$G$-invariant subgraphs]
Given a homomorphism $G \to \Out(A_\Gamma)$ we say that a subgraph $\Sigma \subseteq \Gamma$ is $G$-\emph{invariant} \iff the conjugacy class of $A_\Sigma$ is preserved (setwise) by $G$.
\end{dfn}

\begin{dfn}
Having an invariant subgraph $\Sigma \subseteq \Gamma$ allows us to discuss two additional actions:
\begin{itemize}
 \item Since, for any subgraph $\Sigma$, the normaliser of $A_\Sigma$ in $A_\Gamma$ is equal to $A_\Sigma C(A_\Sigma)$, where $C(A_\Sigma)$ is the centraliser of $A_\Sigma$ (see e.g.~\cite[Proposition 2.2]{Charneyetal2012}), any invariant subgraph $\Sigma$ gives us an \emph{induced (outer) action} $G \to \Out(A_\Sigma)$.
 \item When $\Sigma$ is invariant, we also have the
 \emph{induced quotient action} \[G \to \Out(A_\Gamma / \langle \! \langle A_\Sigma \rangle \! \rangle) \simeq \Out(A_{\Gamma \s- \Sigma})\]
\end{itemize}
\end{dfn}

Let us quote the following.

\begin{lem}[{\cite[Lemmata 4.2 and 4.3]{HenselKielak2016}}]
\label{lem: L}
For any homomorphism $G \to \Out^0(A_\Gamma)$ we have:
\begin{enumerate}
\item for every subgraph $\Sigma \subseteq \Gamma$ which is not a cone, $\lk(\Sigma)$ is $G$-invariant;
\item connected components of $\Gamma$ which are not singletons are $G$-invariant;
\item $\widehat \st(\Sigma)$ is $G$-invariant for every subgraph $\Sigma$;
\item if $\Sigma$ and $\Delta$ are $G$-invariant, then so is $\Sigma \cap \Delta$;
\item if $\Sigma$ is $G$-invariant, then so is $\st(\Sigma)$.
\end{enumerate}
\end{lem}

\begin{dfn}[Trivialised subgraphs]
Let $\phi \colon G \to \Out(A_\Gamma)$ be given.
We say that a subgraph $\Sigma \subseteq \Gamma$ is \emph{trivialised} \iff $\Sigma$ is $G$-invariant, and the induced action is trivial.
\end{dfn}

\begin{lem}
\label{collapsing components}
Let $\phi \colon G \to \Out(A_\Gamma)$ be a homomorphism.
Suppose that $\Sigma$ is a connected component of $\Gamma$ which is trivialised by $G$.
Consider the graph \[\Gamma' = (\Gamma \s- \Sigma) \sqcup \{s\}\] were $s$ denotes a new vertex not present in $\Gamma$. There exists an action \[\psi \colon G \to \Out(A_{\Gamma'})\] for which $\{s\}$ is invariant, and such that the quotient actions
\[ G \to \Out(A_{\Gamma \s- \Sigma})\]
induced by $\phi$ and $\psi$ by removing, respectively, $\Sigma$ and $s$, coincide.
\end{lem}
\begin{proof}
Consider an epimorphism $f \colon A_\Gamma \to A_{\Gamma'}$ defined on vertices of $\Gamma$ by
\[
    f(v) = \left\{ \begin{array}{ccl} v & \textrm{ if } & v \not\in \Sigma \\
                                      s & \textrm{ if } & v \in \Sigma
                    \end{array} \right. \]
The kernel of $f$ is normally generated by elements $vu^{-1}$, where $v,u \in \Sigma$ are vertices.
Since the induced action of $G$ on $A_\Sigma$ is trivialised, the action preserves each element $vu^{-1}$ up to conjugacy. But this in particular means that $G$ preserves the (conjugacy class of) the kernel of $f$, and hence $\phi$ induces an action
\[ G \to \Out(A_{\Gamma'})\]
which we call $\psi$. It is now immediate that $\psi$ is as required.
\end{proof}

\subsection{Finite groups acting on RAAGs}

\begin{dfn}
Suppose that $\Gamma$ has $k$ vertices. Then the abelianisation of $A_\Gamma$ is isomorphic to $\Z^k$, and we have the natural map
\[
\Out(A_\Gamma) \to \Out(H_1(A_\Gamma)) = \GL_k(\Z)
\]
We will refer to the kernel of this map as the \emph{Torelli subgroup}.
\end{dfn}

We will need the following consequence of independent (and more general) results of Toinet~\cite{Toinet2013} and Wade~\cite{Wade2013}.

\begin{thm}[{Toinet~\cite{Toinet2013}; Wade~\cite{Wade2013}}]
\label{torelli}
The Torelli group is torsion free.
\end{thm}

\begin{lem}
\label{killing H}
Let $\phi \colon H \to \Out(A_\Gamma)$ be a homomorphism with a finite domain.
 Suppose that $\Gamma = \Sigma_1 \cup \dots \cup \Sigma_m$, and each $\Sigma_i$ is trivialised by $H$. Then so is $\Gamma$.
\end{lem}
\begin{proof}
Consider the action
\[\psi \colon H \to \Out(H_1(A_\Gamma)) = \GL_k(\Z) \]
obtained by abelianising $A_\Gamma$, where $k$ is the number of vertices of $\Gamma$.
This $\Z$-linear representation $\psi$ preserves the images of the subgroups $A_{\Sigma_i}$, and is trivial on each of them. Thus the representation is trivial, and so $\phi(H)$ lies in the Torelli group. But the Torelli subgroup is torsion free. Hence $\phi$ is trivial.
\end{proof}

\begin{lem}
\label{disconnected case}
Let $\phi \colon G \to \Out(A_\Gamma)$ be a homomorphism.
Let
\[\Gamma = (\Gamma_1 \cup \dots \cup \Gamma_n) \sqcup \Theta\]
where $n \geqslant 1$, each $\Gamma_i$ is trivialised by $G$, and where $\Theta$ is a discrete graph with $m$ vertices.
Suppose that for some $l \in \{m, m+1\}$ any homomorphism
\[ G \to \Out(F_l) \]
is trivial. Then $\Gamma$ is trivialised, provided that $G$ is the normal closure of a finite subgroup $H$, and that $G$ contains a perfect subgroup $P$, which in turn contains $H$.
\end{lem}
\begin{proof}
We can quotient out all of the groups $A_{\Gamma_i}$, and obtain an induced quotient action
\begin{equation}
\label{equation1}
G \to \Out(A_\Theta) \tag{$\ast$}
\end{equation}
We claim that this map is trivial. To prove the claim we have to consider two cases: the first case occurs when $l=m$ in the hypothesis of our lemma, that is every homomorphism 
\[ G \to \Out(F_m) \]
is trivial. Since $\Theta$ is a discrete graph with $m$ vertices, we have $\Out(A_\Theta) = \Out(F_m)$ and so the homomorphism \eqref{equation1} is trivial.

The second case occurs when $l = m+1$ in the hypothesis of our lemma.
In this situation we quotient $A_\Gamma$ by each subgroup $A_{\Gamma_i}$ for $i>1$, but instead of quotienting out $A_{\Gamma_1}$, we use \cref{collapsing components}. This way we obtain an outer action on a free group with $m+1$ generators, and such an action has to be trivial by assumption. Thus we can take a further quotient and conclude again that the induced quotient action \eqref{equation1} on $A_\Theta$ is trivial.
This proves the claim.
\smallskip

Now consider the action of $G$ on the abelianisation of $A_\Gamma$. We obtain a map
\[ \psi \colon G \to \GL_k(\Z) \]
where $k$ is the number of vertices of $\Gamma$. Since each $\Gamma_i$ is trivialised, and the induced quotient action on $A_\Theta$ is trivial, we see that $\psi(G)$ lies in the abelian subgroup of $\GL_n(\Z)$ formed by block-upper triangular matrices with identity blocks on the diagonal, and a single non-trivial block of fixed size above the diagonal. But $P$ is perfect, and so $\psi(P)$ must lie in the Torelli subgroup of $\Out(A_\Gamma)$. This is however torsion free by \cref{torelli}, and so $H$ must in fact lie in the kernel of $\phi$.
We conclude that the action of $G$ on $\Gamma$ is also trivial, since $G$ is the normal closure of $H$.
\end{proof}

\subsection{Some representation theory}

Let us mention a result about representations of $\PSL_n(\Z/p\Z)$, for prime $p$, due to Landazuri and Seitz:
\begin{thm}[\cite{LandazuriSeitz1974}]
\label{thm: landazuri and seitz}
\label{mapfromglnz2}
Suppose that we have a non-trivial, irreducible projective representation $\PSL_n(\Z/p\Z) \to \PGL(V)$, where $n \geqslant 3$, $p$ is prime, and $V$ is a vector space over a field $\K$ of characteristic other than $p$. Then
\[\dim V \geqslant \left\{ \begin{array}{ccc} 2 & \textrm{ if } & (n,p) = (3,2) \\ p^{n-1} -1 & \textrm{ otherwise } \end{array} \right.\]
\end{thm}

We offer an extension of their theorem for algebraically closed fields of characteristic 0, which we will need to discuss actions of $\Out(F_n)$ and $\SOut(F_n)$ on finite sets.
\begin{thm}
\label{thm: reps of GLn(Zq)}
Let $V$ be a non-trivial, irreducible $\K$-linear representation of $\mathrm{SL}_n(\Z/q\Z)$, where $n \geqslant 3$, $q$ is a power of a prime $p$, and where $\K$ is an algebraically closed field of characteristic 0. Then
\[\dim V \geqslant \left\{ \begin{array}{ccc} 2 & \textrm{ if } & (n,p) = (3,2) \\ p^{n-1} -1 & \textrm{ otherwise } \end{array} \right.\]
\begin{proof}
Let $\phi \colon \SL_n(\Z/q\Z) \to \GL(V)$ denote our representation. Consider $Z$, the subgroup of $\SL_n(\Z/q\Z)$ generated by diagonal matrices with all non-zero entries equal. Note that $Z$ is the centre of $\SL_n(\Z/q\Z)$. Hence $V$ splits as an $\SL_n(\Z/q\Z)$-module into intersections of eigenspaces of all elements of $Z$. Since $V$ is irreducible, we conclude that $\phi (Z)$ lies in the centre of $\GL(V)$.

First suppose that $q=p$. Consider the composition
\[\SL_n(\Z/q\Z) \to \GL(V) \to \PGL(V)\]
 We have just showed that $Z$ lies in the kernel of this composition, and so our representation descends to a representation of $\PSL_n(\Z/p\Z) \cong \SL_n(\Z/p\Z)/Z$. This new, projective representation is still irreducible. It is also non-trivial, as otherwise $V$ would have to be a 1-dimensional non-trivial $\SL_n(\Z/q\Z)$-representation. There are no such representations since $\SL_n(\Z/q\Z)$ is perfect when $p=q$. Now Theorem~\ref{thm: landazuri and seitz} yields the result.

Suppose now that $q = p^\alpha$, where $\alpha > 1$.
Let $N \unlhd \SL_n(\Z/q\Z)$ be the kernel of the natural map $\SL_n(\Z/q\Z) \to \SL_n(\Z/p\Z)$. 
As an $N$-module, by Maschke's Theorem, $V$ splits as
\[ V = \bigoplus_{i=1}^k U_i \]
where each $U_i \neq \{0\}$ is a direct sum of irreducible $N$-modules, and irreducible submodules $W \leqslant U_i, W' \leqslant U_j$ are isomorphic \iff $i=j$.

Observe that we get an induced action of $\SL_n(\Z/q\Z) / N \cong \SL_n(\Z/p\Z)$ on the set $\{ U_i , U_2, \ldots , U_k \}$. As $V$ is an irreducible $\SL_n(\Z/q\Z)$-module, the action is transitive.

Note that an action of a group on a finite set $S$ induces a representation on the vector space with basis $S$.
If $k>1$ then this representation is not the sum of trivial ones, because of the transitivity just described, and so
\[k \geqslant \left\{ \begin{array}{ccc} 2 & \textrm{ if } & (n,p) = (3,2) \\ p^{n-1} -1 & \textrm{ otherwise } \end{array} \right. \]
since our theorem holds for $\SL_n(\Z/p\Z)$.
Since $\dim U_i \geqslant 1$ for all $i$, we get $\dim V \geqslant k$ and our result follows.

\smallskip
Let us henceforth assume that $k=1$. We have \[ V=U_1 = \bigoplus_{j=1}^l W \]
where $W$ is an irreducible $N$-module.

Note that we have an alternating group $\Alt_n < \SL_n(\Z/q\Z)$ satisfying \[\Alt_n \cap N = \{ 1\}\] Let $\sigma \in \Alt_n$ be an element of order $o(\sigma)$ equal to 2 or 3.

Consider the group $M = \langle N, \sigma \rangle < \SL_n(\Z/q\Z)$. Note that $M \cong N \rtimes \Z_{o(\sigma)}$. The module $V$ splits as a direct sum of irreducible $M$-modules by Maschke's theorem. Let $X$ be such an irreducible $M$-module.

Note that $X$ as an $N$-module is a direct sum of, say, $m$ copies of the $N$-module $W$ (with $m\geqslant 1$).
Frobenius Reciprocity (see e.g. \cite[Corollary 4.1.17]{weintraub2003}) tells us that the multiplicity $m$ of $W$ (as an $N$-module) in $X$ is equal to the multiplicity of the $M$-module $X$ in the $M$-module induced from the $N$-module $W$. Hence the multiplicity of $W$ in the $M$-module induced from the $N$-module $W$ is at least $m^2$. But it is bounded above by $o(\sigma)$ and $o(\sigma) \leqslant 3$, which forces $m=1$, as $m\geqslant 1$.

This shows in particular that $X$ as  an $N$-module is isomorphic to $W$.
It also shows that the $M$-module induced from $W$ contains a submodule isomorphic to $X$. Since
\[M \cong N \rtimes \Z_{o(\sigma)}\]
 an easy calculation shows that $\sigma$ acts on this copy of $X$ as a scalar multiple of the identity matrix, i.e. via a central matrix. This is true for every irreducible $M$-submodule $X$ of $V$, and hence $\sigma$ commutes with $N$ when acting on $V$.
Since the above statement is true for each $\sigma \in \Alt_n$ of order 2 or 3, we conclude that $\phi$ factors through $\SL_n(\Z/q\Z) / [N,\Alt_n]$.
Note that we need to consider elements $\sigma$ of order 3 when we are dealing with the case $n=4$.

Mennicke's proof of the Congruence Subgroup Property ~\cite{Mennicke1965} tells us that $N$ is normally generated (as a subgroup of $\SL_n(\Z/q\Z)$) by the $p^{th}$ powers of the elementary matrices. Now $\SL_n(\Z/q\Z)$ itself is generated by elementary matrices; let us denote such a matrix by $E_{ij}$ with the usual convention. Observe that for all $\sigma \in \Alt_n$ we have
\[\phi( E_{\alpha \beta}^{-1} E_{ij}^p E_{\alpha \beta} ) = \phi( \sigma^{-1} E_{\alpha \beta}^{-1} E_{ij}^p E_{\alpha \beta} \sigma) = \phi( E_{\sigma(\alpha) \sigma(\beta)}^{-1} E_{ij}^p E_{\sigma(\alpha) \sigma(\beta)} )\]
Choose $\sigma \in A_n$ such that $\sigma(\alpha)=i$ and $\sigma(\beta) = j$. We conclude that $\phi(N)$ lies in the centre of $\phi\big(\SL_n(\Z/q\Z)\big)$. In particular, $\phi(N)$ is abelian, and hence (as $\K$ is algebraically closed) $\dim W =1$, as $W$ is an irreducible $N$-module. Since $V$ is a direct sum of $N$-modules isomorphic to $W$, the group $N$ acts via matrices in the centre of $\GL(V)$.
Hence $N$ lies in the kernel of the composition
\[ \xymatrix{ \SL_n(\Z/q\Z) \ar[r]^\phi & \GL(V) \ar[r] & \PGL(V) } \]
We have already shown that $Z$ lies in this kernel, and so our representation descends to a projective representation of $\PSL_n(\Z/p\Z)$.
If we can show that this representation is non-trivial, we can then apply Theorem~\ref{thm: landazuri and seitz} and our proof will be finished.

Suppose that this projective representation is trivial. This means that $V$ is a 1-dimensional, non-trivial $\SL_n(\Z/q\Z)$-representation. This is however impossible, since the abelianisation of $\SL_n(\Z/q\Z)$ is trivial when $n \geqslant 3$.
\end{proof}
\end{thm}

\subsection{Actions of \texorpdfstring{$\Out(F_n)$}{Out(Fn)} on finite sets}

\begin{thm}
\label{thm: action on finite set}
Every action of $\Out(F_n)$ (with $n \geqslant 6$) on a set of cardinality $m \leqslant \binom {n+1} 2$ factors through $\Z/ 2 \Z$.
\end{thm}
\begin{proof}
Suppose that we are given such an action. It gives us
\[ \Out(F_n) \to \Sym_m \into \GL_{m-1}(\C) \]
where $\Sym_m$ denotes the symmetric group of rank $m$, and the second map is the standard irreducible representation of $\Sym_m$.
Since \[m-1 < \binom{n+1} 2\]
the composition factors through the natural map $\Out(F_n) \to \GL_n(\Z)$ induced by abelianising $F_n$, by \cite[Theorem 3.13]{Kielak2013}. Thus we have
\[ \Out(F_n) \to \GL_n(\Z) \to \GL_{m-1}(\C) \]
with finite image. The Congruence Subgroup Property \cite{Mennicke1965} tells us that the map $\GL_n(\Z) \to \GL_{m-1}(\C)$ factors through a congruence map
\[ \GL_n(\Z) \to \GL_n(\Z / p^\alpha \Z) \]
for some positive integer $\alpha$ and some prime $p$. Now \[m -1 < 2^{n-1} -1 \leqslant p^{n-1} -1\] and so the restricted  map $\SL_n(\Z / p^\alpha \Z) \to \GL_{m-1}(\C)$ must be trivial by \cref{thm: reps of GLn(Zq)}. Thus the given action factors through $\GL_n(\Z / p^\alpha \Z) / \SL_n(\Z / p^\alpha \Z)$, which is an abelian group. Therefore $\SOut(F_n)$ lies in the kernel of $\phi$, since it is perfect (\cref{prop: abelianise sout}), and we are finished.
\end{proof}

\begin{cor}
\label{cor: action on finite set}
Every action of $\SOut(F_n)$ (with $n \geqslant 6$) on a set of cardinality $m \leqslant \frac 1 2 \binom {n+1} 2$ is trivial.
\end{cor}
\begin{proof}
 Every action of an index $k$ subgroup of a group $G$ on a set of cardinality $m$ can be induced to an action of $G$ on a set of cardinality $km$.
\end{proof}

\section{The main result}

\begin{dfn}
Let $D_n$ denote the discrete graph with $n$ vertices.
\end{dfn}

\begin{dfn}
Let $\phi \colon G \to \Out(A_\Gamma)$ be a homomorphism, and let $n$ be fixed.
We define two properties of the action (with respect to $n$):
\begin{enumerate}
 \item[$\mathfrak C$] For every $G$-invariant clique $\Sigma$ in $\Gamma$ with at least $n$ vertices there exists a $G$-invariant subgraph $\Theta$ of $\Gamma$, such that $\Theta \cap \Sigma$ is a proper non-empty subgraph of $\Sigma$.
 \item[$\mathfrak D$] For every $G$-invariant subgraph $\Delta$ of $\Gamma$ isomorphic to $D_n$, there exists a $G$-invariant subgraph $\Theta$ of $\Gamma$, such that $\Theta \cap \Delta$ is a proper non-empty subgraph of $\Delta$.
 \end{enumerate}
\end{dfn}

The notation $\mathfrak C$ stands for `clique', and $\mathfrak D$ for `discrete'.

\begin{lem}
 \label{C and D induced}
Let $\phi \colon G \to \Out(A_\Gamma)$ be an action satisfying $\mathfrak C$ and $\mathfrak D$. Let $\Omega$ be a $G$-invariant subgraph of $\Gamma$. Then both the induced action and the induced quotient action satisfy $\mathfrak C$ and $\mathfrak D$.
\end{lem}
\begin{proof}
Starting with a subgraph $\Sigma$ or $\Delta$ in either $\Omega$ or $\Gamma \s- \Omega$, we observe that the subgraph is a subgraph of $\Gamma$, and so using the relevant property we obtain a $G$-invariant subgraph $\Theta$. We now only need to observe that $\Theta \cap \Omega$ is $G$-invariant by \cref{lem: L}(4), and the image of $\Theta$ in $\Gamma \s- \Omega$ is invariant under the induced quotient action
\[ G \to \Out(A_{\Gamma \s- \Omega}) \qedhere \]
\end{proof}

\begin{thm}
\label{ass D}
Let us fix positive integers $n$ and $m\geqslant n$.
Suppose that a group $G$ satisfies all of the following:
\begin{enumerate}
 \item $G$ is the normal closure of a finite subgroup $H$.
 \item All homomorphisms
\[ G \to \Out(F_k) \]
are trivial when $k \neq n$ and $k < m$.
 \item All homomorphisms
\[ G \to \GL_k(\Z)\]
are trivial when $k<n$.
\item Any action of $G$ on a set of cardinality smaller than $m$ is trivial.
\end{enumerate}
Let \[\phi \colon G \to \Out(A_\Gamma)\] be a homomorphism, where $\Gamma$ has fewer than $m$ vertices.
Then $\phi$ is trivial, provided that the action satisfies properties $\mathfrak C$ and $\mathfrak D$ (with respect to $n$).
\end{thm}
\begin{proof}
Formally, the proof is an induction on the number of vertices of $\Gamma$, and splits into two cases.

Before we proceed, let us observe that assumption (4) allows us to apply \cref{out0 for homs}, and hence to use \cref{lem: L} whenever we need to.

\medskip
\noindent \textbf{Case 1:} Suppose that $\Gamma$ does not admit proper non-empty $G$-invariant subgraphs.

Note that this is in particular the case when $\Gamma$ is a single vertex, which is the base case of our induction.

We claim that $\Gamma$ is either discrete, or a clique.
To prove the claim, let us suppose that $\Gamma$ is not discrete.

Let $v$ be a vertex of $\Gamma$ with a non-empty link. \cref{lem: L}(3) tells us that $\widehat \st(v)$ is $G$-invariant, and thus it must be equal to $\Gamma$. Hence $\Gamma$ is a join, and therefore admits a join decomposition.

If each factor of the decomposition is a singleton, then $\Gamma$ is a clique as claimed. Otherwise, the decomposition contains a factor $\Sigma$ which is not a singleton and not a join, and so in particular not a cone. Thus \cref{lem: L}(1) informs us that $\lk(\Sigma)$ is $G$-invariant. This is a contradiction, since this link is a proper non-empty subgraph.
We have thus shown the claim.

\smallskip

Suppose that $\Gamma$ is a clique, with, say, $k$ vertices.
Property $\mathfrak C$ immediately tells us that $k<n$, and so
we are dealing with a homomorphism
\[ \phi \colon G \to \Out(A_\Gamma) = \GL_k(\Z) \]
where $k<n$. Such a homomorphism is trivial by assumption (3).

\smallskip

Suppose that $\Gamma$ is a discrete graph, with, say, $k$ vertices.
Property $\mathfrak D$ immediately tells us that $k\neq n$, and so
we are dealing with a homomorphism
\[ \phi \colon G \to \Out(A_\Gamma) = \Out(F_k) \]
where $k\neq n$ and $k<m$. Such a homomorphism is trivial by assumption (2).

\medskip
\noindent \textbf{Case 2:} Suppose that $\Gamma$ admits a proper non-empty $G$-invariant subgraph $\Sigma$.

\cref{C and D induced} guarantees that the induced action
\[ G \to \Out(A_\Sigma) \]
satisfies the assumptions of our theorem, and thus, using the inductive hypothesis, we conclude that this induced action is trivial.

We argue in an identical manner for the induced quotient action
\[ G \to \Out(A_{\Gamma \s-\Sigma}) \]
and conclude that it is also trivial.

These two observations imply that in particular the restriction of these two actions to the finite group $H$ from assumption (1) is trivial. Now \cref{killing H} tells us that $H$ lies in the kernel of $\phi$, and hence so does $G$, as it is a normal closure of $H$ by assumption (1).
\end{proof}

\begin{lem}
 \label{replace ass C}
 Suppose that $\Gamma$ does not contain $n$ distinct vertices with identical stars. Then property $\mathfrak C$ holds for any action $G \to \Out^0(A_\Gamma)$.
\end{lem}
\begin{proof}
 Let $\Sigma$ be a $G$-invariant clique in $\Gamma$ with at least $n$ vertices. Since we know that no $n$ vertices of $\Gamma$ have identical stars, we need to have distinct vertices of $\Sigma$, say $v$ and $w$, with $\st(v) \neq \st(w)$. Without loss of generality we may assume that there exists $u \in \st(v) \s- \st(w)$.
In particular this implies that $u$ and $w$ are not adjacent.

Consider $\Lambda= \lk(\{u,w\})$: it is invariant
by \cref{lem: L}(1), since $\{u,w\}$ is not a cone; it intersects $\Sigma$ non-trivially, since the intersection contains $v$; the intersection is also proper, since $w \not\in \Lambda$.
Thus property $\mathfrak C$ is satisfied.
\end{proof}

\begin{prop}
\label{replace ass D}
 In \cref{ass D}, we can replace the assumption on the action satisfying $\mathfrak D$ by the assumption that $\Gamma$ is not a join of $D_n$ and another (possibly empty) graph, provided that $G$ satisfies additionally
 \begin{enumerate}
  \item[(5)] $G$ contains a perfect subgroup $P$, which in turn contains $H$.
 \end{enumerate}
\end{prop}
\begin{proof}
We are going to proceed by induction on the number of vertices of $\Gamma$, as before.
Assuming the inductive hypothesis, we will either show the conclusion of the theorem directly, or we will show that in fact property $\mathfrak D$ holds.

Note that the base case of induction ($\Sigma$ being a singleton) always satisfies $\mathfrak D$.

Let $\Delta$ be as in property $\mathfrak D$, and suppose that the property fails for this subgraph.

\medskip \noindent \textbf{Case 1:} suppose that there exists a vertex $u$ of $\Delta$ with a non-empty link.

Let $v$ be a vertex of $\Gamma \s- \Delta$ joined to some vertex of $\Delta$.
Consider $\widehat \st(v)$; this subgraph is $G$-invariant by \cref{lem: L}(3). If $\widehat \st(v)$ intersects $\Delta$ and does not contain it, then $\Delta$ does satisfy property $\mathfrak D$. We may thus assume that $\Delta \subseteq \widehat \st(v)$.

 We would like to apply induction to $\widehat \st(v)$, and conclude that this subgraph, and hence $\Delta$, are trivialised. This would force $\Delta$ to satisfy property $\mathfrak D$.

There are two cases in which we cannot apply the inductive hypothesis to $\widehat \st(v)$: this subgraph might be equal to $\Gamma$, or it might be a join of a subgraph isomorphic to $D_n$ and another subgraph.

\smallskip
In the former case, $\Gamma$ is a join of two non-empty graphs. If there exists a factor $\Theta$ of the join decomposition of $\Gamma$ which is not a singleton, and which does not contain $\Delta$, then let us look at $\lk(\Theta)$. This is a proper subgraph of $\Gamma$, it is $G$-invariant by \cref{lem: L}(1), and is not a join of $D_n$ and another graph since $\Gamma$ is not. Thus we may apply the inductive hypothesis to $\lk(\Theta)$ and conclude that it is trivialised. But $\Delta \subseteq \lk(\Theta)$, and so $\Delta$ is also trivialised, and thus satisfies $\mathfrak D$.

If $\Gamma$ has no such factor $\Theta$ in its join decomposition, then $\Gamma = \st(\Sigma)$, where $\Sigma$ is a non-empty clique. The clique $\Sigma$ is a proper subgraph, since it does not contain $\Delta$. It is $G$-invariant by \cref{lem: L}(1)  and so the inductive hypothesis tells us that it is trivialised.

The induced quotient action $G \to \Out(A_{\Gamma \s- \Sigma})$ is also trivialised by induction, as $\Gamma \s- \Sigma$ cannot be a join of $D_n$ and another graph as before. We now apply \cref{killing H} for the subgroup $H$, and conclude that $H$, and hence its normal closure $G$, act trivially.

\smallskip
Now we need to look at the situation in which $\widehat \st(v)$ is a proper subgraph of $\Gamma$, but it is a join of $D_n$ and another graph.

Let us look at $\Lambda$, the intersection of $\widehat \st(v)$ with the link of all factors of the join decomposition of $\widehat \st(v)$ isomorphic to $D_n$. The subgraph $\Lambda$ is $G$-invariant by \cref{lem: L}(1) and (4). It is a proper subgraph of $\Gamma$, and so the inductive hypothesis tells us that $\Lambda$ is trivialised. If $\Lambda$ contains $\Delta$ then we are done.

The graph $\Lambda$ does not contain $\Delta$ \iff $\Delta$ is a factor of the join decomposition of $\widehat \st(v)$.
Observe that we can actually use another vertex of $\Gamma \s- \Delta$ in place of $v$, provided that this other vertex is joined by an edge to some vertex of $\Delta$. Thus we may assume that $\Delta$ is a factor of the join decomposition of every $\widehat \st(v)$ where $v$ is as described. This is however only possible when $\st(\Delta)$ is a connected component of $\Gamma$.
There must be at least one more component, since $\Gamma$ is not a join of $\Delta$ and another graph.

Note that the component $\st(\Delta)$ is invariant by \cref{lem: L}(5).

Suppose that the clique factor $\Sigma$ of $\lk(\Delta)$ is non-trivial. As before, $\Sigma$ is trivialised.
Observing that $\Gamma \s- \Sigma$ is disconnected, and if it is discrete then it is has more than $n$ vertices, allows us to apply the inductive hypothesis to the quotient action induced by $\Sigma$, and so, arguing as before, we see that $\Gamma$ is trivialised.

Now suppose that $\lk(\Delta)$ has a trivial clique component. The join decomposition of the component $\st(\Delta)$ consists of at least two factors, each of which is invariant by \cref{lem: L}(1).
Let $\Theta$ be such a factor. Removing $\Theta$ leaves us with a disconnected graph smaller than $\Gamma$. Thus, we may apply the inductive hypothesis, provided that $\Gamma \s- \Theta$ is not $D_n$. This might however occur: in this situation $\st(\Delta) \s- \Theta$ fulfils the role of the graph $\Theta$ from the definition of $\mathfrak D$, and so we can use the inductive hypothesis nevertheless.

We now apply \cref{killing H} to the subgroup $H$ and the induced quotient actions determined by removing two distinct factors of ${\st(\Delta)}$, and conclude that $H$, and hence its normal closure $G$, act trivially on $A_{\Gamma}$.

\medskip \noindent \textbf{Case 2:} $\lk(u) = \emptyset$ for every vertex $u$ of $\Delta$.

We write $\Gamma = \Gamma_1 \sqcup \dots \sqcup \Gamma_k \sqcup \Theta$ where the subgraphs $\Gamma_i$ are non-discrete connected components of $\Gamma$, and $\Theta$ is discrete. By assumption $\Delta \subseteq \Theta$.

If $k \geqslant 2$, then removing any component $\Gamma_i$ leaves us with a smaller graph, to which we can apply the inductive hypothesis. Then we use \cref{disconnected case}.

If $k=0$ then $\Theta$ is not isomorphic to $D_n$ by assumption. Then we know that the action $\phi$ is trivial by assumption (2).

If $k=1$, then we need to look more closely at $\Gamma_1$. If $\Gamma_1$ does not have factors isomorphic to $D_n$ in its join decomposition, then by induction we know that $\Gamma_1$ is trivialised. Now we use \cref{disconnected case}.

Suppose that $\Gamma_1$ contains a subgraph $\Omega$ isomorphic to $D_n$ in its join decomposition.
If $\Gamma_1$ has a non-trivial clique factor, then this factor is invariant, induction tells us that it is trivialised, and the induced quotient action is also trivial. Thus the entire action of $H$ is trivial, thanks to \cref{killing H}, and thus the action of $G$ is trivial, as $G$ is the normal closure of $H$.

If the clique factor is trivial, then taking links of different factors of the join decomposition of $\Gamma_1$ allows us to repeat the argument we just used, and conclude that $H$, and thus $G$, act trivially.
\end{proof}

\begin{thm}
\label{main thm}
Let $n \geqslant 6$. Suppose that $\Gamma$ is a simplicial graph with fewer than $\frac 1 2 \binom n 2$ vertices. Let $\phi \colon \SOut(F_n) \to \Out(A_\Gamma)$ be a homomorphism.
Then $\phi$ is trivial, provided that there are no $n$ vertices in $\Gamma$ with identical stars, and that $\Gamma$ is not a join of the discrete graph with $n$ vertices and another (possibly empty) graph.
\end{thm}
\begin{proof}
We start by showing that $G = \SOut(F_n)$ satisfies the assumptions (1)--(4) of \cref{ass D} and (5) of \cref{replace ass D}, with $m = \frac 1 2 \binom n 2$.
\begin{enumerate}
 \item Let $H = \Alt_n$. The group $G$ is the normal closure of $H$ by \cref{prop: A_n in SOut}.

 \item All homomorphisms
 \[ G \to \Out(F_k)\]
 are trivial when $k \neq n$ and $k < m$ by \cref{thm: main phd for sout}.
 \item All homomorphisms
 \[ G \to \GL_k(\Z)\]
 are trivial when $k < n$ by \cref{thm: maps from out to small gl}.
 \item Any action of $G$ on a set of cardinality smaller than $m$ is trivial by \cref{cor: action on finite set}.
 \item $G$ is perfect by \cref{prop: abelianise sout}.
\end{enumerate}

To verify property $\mathfrak C$ we use \cref{replace ass C}, and property $\mathfrak D$ we replace using \cref{replace ass D}. Now we apply \cref{ass D}.
\end{proof}

\section{From larger to smaller RAAGs}

In this section we will look at homomorphisms $\Out(A_\Gamma) \to \Out(A_{\Gamma'})$, where $\Gamma'$ has fewer vertices than $\Gamma$.

\begin{thm}
\label{main fewer vertices}
 There are no injective homomorphisms $\Out(A_\Gamma) \to \Out(A_{\Gamma'})$ when $\Gamma'$ has fewer vertices than $\Gamma$.
\end{thm}
\begin{proof}
For a group $G$ we define its $\Z_2$-\emph{rank} to be the largest $n$ such that ${(\Z_2)}^n$ embeds into $G$.

We claim that
the $\Z_2$-rank of $\Out(A_\Gamma)$ is equal to $\vert \Gamma \vert$, the number of vertices of $\Gamma$.

Firstly, note that
for every vertex of $\Gamma$ we have the corresponding inversion in $\Out(A_\Gamma)$, and these inversions commute; hence the $\Z_2$-rank of $\Out(A_\Gamma)$ is at least $\vert \Gamma \vert$.

For the upper bound, observe that the $\Z_2$-rank of $\GL_n(\R)$ is equal to $n$, since
we can simultaneously diagonalise commuting involutions in $\GL_n(\R)$. Thus, the $\Z_2$-rank of $\GL_n(\Z)$ is equal to $n$ as well (since it is easy to produce a subgroup of this rank).

Finally, note that the kernel of the natural map $\Out(A_\Gamma) \to \GL_n(\Z)$ with $n = \vert \Gamma \vert$ is torsion free by
\cref{torelli}, and so the $\Z_2$-rank of $\GL_n(\Z)$ is bounded below by the $\Z_2$-rank of $\Out(A_\Gamma)$.
\end{proof}

\begin{rmk}
 The proof of the above theorem works for many subgroups of $\Out(A_\Gamma)$ as well; specifically it applies to $\Out^0(A_\Gamma)$, the group of \emph{untwisted} outer automorphisms $\mathrm{U}(A_\Gamma)$, and the intersection $\mathrm{U}^0(A_\Gamma) = \mathrm{U}(A_\Gamma) \cap \Out^0(A_\Gamma)$.
 
It also works when the domain of the homomorphisms is $\Aut(A_\Gamma)$, or more generally any group with $\Z_2$-rank larger than the number of vertices of $\Gamma'$.
\end{rmk}

\bibliographystyle{math}
\bibliography{bibliography}

\bigskip

\noindent Dawid Kielak \newline
Fakult\"at f\"ur Mathematik  \newline
Universit\"at Bielefeld \newline
Postfach 100131  \newline
D-33501 Bielefeld \newline
Germany \newline
\texttt{dkielak@math.uni-bielefeld.de}

\end{document}